\documentclass[11pt,reqno]{amsart}
\usepackage{amsmath,amssymb,rawfonts}
\usepackage{tikz}
\usepackage[english]{babel}
\usepackage{pb-diagram}
\usepackage{graphicx}
\usepackage[utf8]{inputenc}
\newcommand{\be}{\begin{equation}}
\newcommand{\ee}{\end{equation}}

\newtheorem{theorem}{Theorem}[section]
\newtheorem{lemma}[theorem]{Lemma}

\newtheorem{prop}[theorem]{Proposition}
\newtheorem{coro}[theorem]{Corollary}

\newtheorem{ex}[theorem]{Example}

\begin{document}  
\title{  RELATED REFLECTIONS  TO THE AXIOMS OF SEPARATION IN SEMIGROUPS WITH TOPOLOGIES AND SOME APPLICATIONS.}
\author{JULIO CÉSAR HERNÁNDEZ ARZUSA}

\maketitle 

\begin{abstract}
In this paper we study the reflections of the  category of topological and semitopological semigroups on the category of the class of topological spaces satisfying separation axioms $T_{0}$, $T_{1}$, $T_{2}$, $T_{3}$ and regular  and we apply its properties  for to find conditions under which a topological semigroup has the Souslin property.\\

\end{abstract}


\section*{\textbf{INTRODUCTION}.} Let $\mathcal{C}$ be an epireflective category of topological spaces (in what follows $TOP$) and let $\mathcal{C}$ be the  functor associated with $\mathcal{C}$. In \cite{HH} it is proven that each semitopological (separalaty continuous) algebraic  structure in $X\in Top$ is reflected in $\mathcal{C}(X)$, also it is proven that if the functor $\mathcal{C}$ respects finite products, then the topological (jointly continouos) algebraic  structures are reflected. In \cite{T4} we can find similar results, but only in semitopological groups. In this work we studied  the  functors  related to axioms of separation  in the category of topological and semitopological semigroups. Not having  inverse in the semigroup operation gives us certain problems, for example, one of them is that not always the quotient mappings induced  by congruences are open, what happens in semitopological groups. For this reason we have considered in ours results topological semigroups with open shifts, these are studied in \cite{B} by A. Ravsky and T. Banakh.\\
The fact that the classes of the $T_{0}$, $T_{1}$, $T_{2}$, $T_{3}$, Regular and Tychonoff spaces,  constitute  epiriflective subcategories  of $TOP$, is proven in  \cite{herl},  in \cite{T2} and \cite{T4}  are studied the related functors to  these reflections in  the category of the semitopological  and paratopological groups.\\

 We have developed this paper in four sections. \\In the first section we give the basic facts of the theory and the basic  notation. Additionally we give examples of topological (semitopological) semigroups with open shifts, in order to prove that this class contains the class of the semitopological (resp. paratopological) groups as a proper subclass.\\ In the second section,  we initially  construct the reflection on $T_{2}$ spaces  of semitopological monoids with open shifts (Theorem \ref{513}). In relation to $T_{2}$ semigroup we gived a partial answer (Corollary \ref{04152}) to the following  question: If $S$ is a topological semigroup and $\sim$ is a closed  congruence on $S$, is $S/\sim$ a topological semigroup? This problem is addressed in \cite{GONZO} and \cite{K}, but unlike these works, in our  results the condition being $T_{2}$ is not assumed for semitopological semigroups. Moreover,  we give a characterization of the reflection on $T_{3}$ spaces for topological monoids with open shifts(Corollary \ref{075}), where we proved the topological monoid with open shifts are quasiregular spaces, obtaining a more general result than  Theorem 2.4 in \cite{T2}. Finally we construct the reflection on regular spaces for topological monoids with open shifts (Theorem \ref{0153}). \\In the third section we find results about  epireflections that preserve pro\-ducts in the category of semigroups. Analogues results are found in \cite{T3}, but in the category of groups. 
 \\In the fourth section we use the $\mathcal{C}$-reflections  to study the cellularity of cancellative topological  monoids with open shifts. Given $\mathcal{C}_{r}$-reflection preserves the cellularity (Theorem \ref{0751}) we can obtain some  results releted to the Souslin property, without using separation axioms. In \cite{T5} M. Tkachenko proved each $\sigma$-compact topological group has countable cellularity, this result was generalized by V. Unspenskij in \cite{U}. Later  A. Arhalhensky and E. Reznychenko  extended it for Hausdorff $\sigma$-compact paratopological groups in \cite{A}. Finally M. Tkachenko proved that   the  $\sigma$-compact paratopological groups have countable cellularity. We present similar results for compact cancellative topological   monoids with open shifts (Theorem \ref{5132}), and $\sigma$-compact and  sequentially compact cancellative topological   monoids with open shifts (Theorem \ref{5131}).\\
 The  $T_{0}$, $T_{1}$, $T_{2}$, $T_{3}$, regular and Tychonoff spaces are defined according to \cite{T2}. 

\section{\textbf{PRELIMINARIES}}
 A \textit{semigroup} is a set $S\neq \emptyset$, endowed with an associative operation. If also $S$ has neutral element, we say that $S$ is a \textit{monoid}. A mapping $f\colon S\longrightarrow H$ between semigroups is a \textit{homomorphism} if $f(xy)=f(x)f(y)$ for all $x,y\in S$ . A \textit{semitopological semigroup} (monoid) consists of a semigroup (resp. monoid) $S$ and a topology $\tau$ on $S$ such that for all $a\in S$, the shifts $x\mapsto ax$ and $x\mapsto xa$ (noted by $l_{a}$ and $r_{a}$, respectively)  are continuous mapping of the $S$ to itself. We say that a semitopological semigroup has \textit{open shifts}, if for each $a\in S$ and for each open set  $U$ in $S$, we have $l_{a}(U)$ and $r_{a}(U)$ are open sets in $S$. A \textit{topological semigroup} (monoid)(\textit{paratopological group}) consists of a semigruop (resp. monoid)(resp. group) $S$ and a topology $\tau$ such that the operation of $S$, as a mapping of $S\times S$ to $S$ is continuous, when $S\times S$ is endowed with the product topology. An \textit{congruence} on a semigrouop $S$ is an equivalence relation on $S$, $\sim$, such that if $x\sim y$ and $a\sim b$, then $xa\sim yb$. If $S$ is a semitopological semigroup, then we say that $\sim$ is a \textit{closed congruence} if $\sim$ is closed in $S\times S$.  If $\sim$ is an equivalence relation in a semigroup (monoid) $S$ and $\pi\colon S\longrightarrow S/\sim$ is the respective quotient mapping, then $S/\sim$ is a semigroup (monoid) and $\pi$ an homomorphism  if and only if $\sim$ is a congruence (\cite{GONZO}, Theorem 1).\\ A  class $C$ of topological spaces, is called \textit{closed under super topologies} if  $(X,\tau)\in C$ implies $(X,\rho)\in C$ for each topology  $\rho$ on $X$ finer than $\tau$.\\ Let $X$ be a topological space, a \textit{cellular family} in $X$ is a non empty family of non empty open sets in $X$  and pairwise   disjoint. The \textit{cellularity} of a space $X$, noted by $c(X)$, is defined by $$c(X)=sup\{|U|: U \mbox{ is cellular familiy in $X$}\}+\aleph_{0}. $$ If $c(X)=\aleph_{0}$, we say that $X$ has \textit{countable cellularity} or $X$ \textit{has the Souslin property}.\\If $X$ is a topological space and $A\subseteq X$. We will note by $Int_{X}(A)$ and $Cl_{X}(A)$, the interior and the closure of $A$ in $X$, or simply $Int(A)$ and $\overline{A}$, respectively,  when the space $X$ is understood. An open set $U$ in $X$, is called \textit{regular open} in $X$ if $Int\overline{U}=U$. It is easy to prove the regular open ones form a base for a topology, which we will call semiregularitation of $X$, $X$ endowed with this topology, we will note by $X_{sr}$. $X$ is called \textit{quasiregular} if $X_{sr}$ is a $T_{3}$ space.\\
If $\mathcal{C}$ is an epireflective class of $TOP$, $X$ is a topological space and the morphism $r\colon X \longrightarrow B$ is the $\mathcal{C}$-reflection of $X$, then, given the reflections are essentially   unique, in order to agree with the notation,  we will note $r$ by $\varphi_{(\mathcal{C},X)}$ and $B$ by $\mathcal{C}(X)$. The functor induced by the $\mathcal{C}$-reflection, we will note it by $\mathcal{C}$, therefore if $f\colon X\longrightarrow Y$ is a continuous mapping, there is an unique continuous mapping $\mathcal{C}(f)\colon \mathcal{C}(X)\longrightarrow \mathcal{C}(Y)$, such that $\mathcal{C}(f)\circ \varphi_{(\mathcal{C},X)}=\varphi_{(\mathcal{C},Y)}\circ f$.
 $\mathcal{C}_{0}$, $\mathcal{C}_{1}$, $\mathcal{C}_{2}$, $\mathcal{C}_{3}$, $\mathcal{C}_{r}$ y $\mathcal{C}_{t}$ will note the class of the spaces, $T_{0}$, $T_{1}$, $T_{2}$, $T_{3}$, regular and Tychonoff, respectively.\\ A topological space $X$ is called \textit{$\sigma$-compact} if it is countable union of compact subsets. $X$ is called \textit{sequentially compact} if each sequence in $X$ has a subsequence converging in $X$. \\\\
The following examples guarantee the class of cancellative topological  monoids with open shifts in non empty and is bigger than the class of paratopological groups.

\begin{ex} Let $\mathbb{R}^{+}=[0,\infty)$, together with the usual sum in $\mathbb{R}$, endowed with the generated topology by the sets $[a,\infty)$, being $a\in \mathbb{R}^{+}$. Then $\mathbb{R}^{+}$ is a compact cancellative topological  monoid with open shifts and it is not a group.
\end{ex}
\begin{ex}Let $G$ be a paratopological group that is not a topological group (for example the Sorgenfrey line) and let $U$ be an open non symmetric ($U^{-1}\neq U$) neighborhood in $G$ of the neutral element $e$ of $G$. If $S=\bigcup_{n\in \mathbb{N}}U^{n}$, then $S$ is a cancellative topological  monoid with open shifts that is not a group.
\end{ex}
\begin{ex}\label{516}It is possible to obtain open shifts from semitopological (topological) monoids. Indeed  let $S$ be  a semitopological (topological) monoid and let $N_{e}$ be  an open local base of the neutral element $e$ of $S$. The set $\gamma=\{aU:U\in N_{e},a\in S\}\cup \{Ua:U\in N_{e},a\in S\}$ generates a topology of semitopological (resp. topological) monoid with open shifts.
\end{ex}

\begin{ex}Let $S=\{(x,y)\in \mathbb{R}^{2}: y\geq 0\}$ endowed with the usual sum, and with the generated topology by the family $\{V_{x}\}_{x \in \mathbb{R}}$, where $V_{x}=\{(x,y): y\geq 0\}$. Then $S$ is a cancellative topological  monoid with open shifts. Moreover, $S$ is not $T_{0}$ space.
\end{ex}

\begin{ex}Let $S$ be an infinite  cancellative  semigroup, if $S$ is endowed with the cofinite topology, then $S$ is a cancellative semitopological semigroup with open shifts.
\end{ex}

The following result clarifies the action the epi-reflection functor for subcategories closed under supertopologies. (see \cite{herlG} for the proof, which is
straightforward anyway).

\begin{prop}\label{0131}Let $\mathcal{C}$ be an epireflective class of $TOP$. Then $\mathcal{C}(X)$ is a quotient of $X$ if and only if $\mathcal{C}$ is closed under super topologies.
\end{prop}

The Theorem 3.4 in \cite{T4} allows us to obtain a inner characterization of the $\mathcal{C}_{1}$-reflection in the category of semitopological groups, the following corollary is more general.

\begin{prop}\label{514}Let  $\mathcal{C}$ be  an epireflective subcategory of $TOP$ closed under super topologies. Then $\mathcal{C}(X)=X/R_{\mathcal{C}}$, where $$R_{\mathcal{C}}=\bigcap\{R: \mbox{ $R$  is an equivalence relation and } X/R\in \mathcal{C}\}.$$
\end{prop}

\begin{proof}
From Proposition \ref{0131}, $\mathcal{C}(X)=X/R_{1}$, being $R_{1}$ an equivalence relation in $X$. If $\pi\colon X\longrightarrow X/R_{\mathcal{C}}$ is the respective quotient mapping, since $R_{\mathcal{C}}\subseteq R_{1}$, the mapping $f\colon X/R_{\mathcal{C}}\longrightarrow X/R_{1}$, defined by $f(\pi(x))=\varphi_{(\mathcal{C}, X)}$ is well defined, bijective and quotient, therefore it  is homeomorphism, so that $R_{1}=R_{\mathcal{C}}$, this completes the proof.
\end{proof}

We can find the proof of following proposition in \cite{ACOSTA}, where the reflection on the $T_{0}$ spaces is called the $T_{0}$-identification.

\begin{prop}\label{0152}Let $X$ be a topological space. Then $\varphi_{(\mathcal{C}_{0},X)}^{-1}(\varphi_{(\mathcal{C}_{0},X)}(U))=U$ for each closed or open set in $X$, therefore $\varphi_{(\mathcal{C}_{0},S)}$ is open and closed mapping. Also $\varphi_{(\mathcal{C}_{0},X)}(x)\neq \varphi_{(\mathcal{C}_{0},X)}(y)$ if and only there exist $U$, open in $X$, such that $U\cap \{x,y\}$ is a singleton.
\end{prop}

The following theorem appears in \cite{HH} for algebraic structures more general, we  give a similar  proof for semigroups.

\begin{theorem} \label{0156}Let $S$ be a semitopological semigroup (monoid) and $\mathcal{C}$ an epireflective class of $TOP$. Then $\mathcal{C}(S)$ is a semitopolgical  semigroup (resp. monoid) and $\varphi_{S}$ is a homomorphism.
\end{theorem}

\begin{proof}
For each $a\in S$, the continuous   mappings $l_{a}$ and $r_{a}$ allow to define continuous mappings $\mathcal{C}(l_{a})$ and $\mathcal{C}(r_{a})$ from $\mathcal{C}(S)$ to itself, for $\mathcal{C}(l_{a})((\varphi_{S}(x))=\varphi_{S}(ax)$ and $\mathcal{C}(r_{a})((\varphi_{S}(x))=\varphi_{S}(xa).$  Therefore the operation on $\mathcal{C}(S)$ defined by $\varphi_{S}(x)\varphi_{S}(y)=\varphi_{S}(xy)$ is well defined and also $\varphi_{S}$ is an homomorphism.
\end{proof}

\begin{prop}\label{035}Let $S$ be a semitopological semigroup (monoid) and let $\sim$ be a congruence in $S$ and let $\pi\colon S\longrightarrow S/\sim$ the respective quotient mapping. Then $S\sim$ is a semitopological semigroup (resp. monoid) and $\pi$ is a  homomorphism. Also if $S$ is a topological semigroup and $\pi\times \pi$ is quotient mapping, then $S/\sim$ is a topological semigroup (resp. monoid). In particular, if $\pi$ is open and $S$ is a topological semigroup, then the same is true for $S/\sim$.
\end{prop}

\begin{proof}
Let $S$ be a semitopological semigroup, if $\sim$ is a congruence, obviously the operation defined by $\pi(x)\ast\pi(y)=\pi(xy)$, for each $x,y\in S$ is well defined and associative on $S/\sim$, therefore  $(S/\sim,\ast)$ is a semigroup and $\pi$ is a homomorphism. If $S$ is monoid and $e$ is its neutral element, then $\pi(e)$ is the neutral element in $S/\sim$, therefore $S/\sim$ is a monoid. Since $\pi$ is quotient mapping, we have $\ast$ is separately continuous, and  in consequence we have $S/\sim$ is a semitopological semigroup. If $S$ is a topological semigroup and $\pi\times \pi$ is a quotient mapping, then the continuity of the operation on $S$ implies that $\ast$ is continuous, therefore $S/\sim$ would be a topological semigroup.
\end{proof}

Since if $X$ es $T_{3}$, then $\mathcal{C}_{0}(X)$ is $T_{3}$ (From Proposition \ref{0152} $\varphi_{(\mathcal{C}_{0}, X)}$ is open and closed mapping), the proof of the following proposition is trivial.

\begin{prop}\label{0151}For each topological space $X$, we have $\mathcal{C}_{r}(X)=\mathcal{C}_{0}(\mathcal{C}_{3}(X).$
\end{prop}

\section{\textbf{RELATED FUNCTORS TO AXIOMS OF SEPARATIONS IN SEMIGROUPS}.}
\begin{prop}\label{0415} Let $S$ be a semitopological monoid where right shifts or left shifts are open, and let $\sim$ an congruence on $S$. Then the respective quotient mapping $\pi \colon S\longrightarrow S/\sim$ is open.
\end{prop}
\begin{proof}Only we will prove the statement when the left shifts are open, the right case is analogue.
Proving what $\pi$ is open, we will proof what $\pi^{-1}(\pi(U))$ is open in $X$ for each $U$ open in $X$. Indeed let $x$ be in $\pi^{-1}(\pi(U))$, where $U$ is open in $X$. Hence there is $u\in U$, such that $\pi(x)=\pi(u)$. Since $l_{u}(e)=u$, we can find a neighborhood of $e$, $V$, such that $uV\subseteq U$. We will prove that $xV\subseteq \pi^{-1}(\pi(U))$, this would prove that $x\in Int(\pi^{-1}(\pi(U))$, and therefore $\pi^{-1}(\pi(U))$ would be open. Let $t\in xV$, then $t=xv$, where $v\in V$. Since $\sim$ is a congruence, we have that $\pi(t)=\pi(xv)=\pi(uv)\subseteq \pi(U)$, therefore $t=xv\in \pi^{-1}(\pi(U)$, this completes the proof.
\end{proof}

Since the class of the $T_{2}$, $T_{1}$ and $T_{0}$ space are closed for supertopologies, then from  propositions \ref{0131}, \ref{0156}, \ref{035} and \ref{0415}, we have the following corollary.

\begin{coro}\label{0354} If  $S$ is a topological monoid with open shifts, then $\mathcal{C}_{i}(S)$ is a monoid for each $i\in\{0,1,2\}$.
\end{coro}

\begin{prop}\label{04151}Let $X$ be a topological space and $\sim$ a equivalence relation on $X$. Then, if $X/\sim$ is $T_{2}$, $\sim$ is closed $X\times X$. The reciprocal holds if the quotient mapping $\pi\colon X\longrightarrow X/\sim$ is open.
\end{prop}

\begin{proof}
Let us suppose that $(x,y)\notin \sim$, then $\pi(x)\neq \pi(y)$, if $X/\sim$ is $T_{2}$, there are open disjoint neighborhoods  $U_{\pi(x)}$ and  $U_{\pi(y)}$  of $\pi(x)$ and $\pi(y)$ in $X/\sim$, respectively. The continuity of $\pi$ guarantees that there are neighborhood $V_{x}$, $V_{y}$, of $x$ and $y$ in $X$, respectively,  such that $\pi(V_{x})\subseteq U_{\pi(x)}$ and $\pi(U_{y})\subseteq U_{\pi(y)}$. It follows that $(V_{x}\times V_{y})\cap \sim=\emptyset.$ This proves that $\sim$ is closed in $X\times  X$. Reciprocally, let us suppose that $\pi$ is open and $\sim$ is closed and let us prove $X/\sim$ is $T_{2}$. Indeed let us $\pi(x)\neq \pi(y)$, where $x,y\in X$. Since $\pi(x)\neq \pi(y)$ and $\sim$ is closed, we can find open sets in $X$, $V_{x}$ and $V_{y}$, containing to $x$ and $y$, respectively, such that $(V_{x}\times V_{y})\cap \sim=\emptyset$. It follows that $\pi(V_{x})\cap \pi(V_{x})=\emptyset$. but $\pi(V_{x})$ and $\pi(V_{y})$ are neighborhoods of $\pi(x)$ and $\pi(y)$, respectively, this proves that $X/\sim$ is $T_{2}$.
\end{proof}

The next corollary easily follows from propositions \ref{0415} and \ref{04151}

\begin{coro} \label{04152}Let $S$ be a semitopological monoid with left shift s or right open shifts  and let $\sim$ be a congruence on $S$. Then $S/\sim$ is $T_{2}$ if and only if $\sim$ is a closed congruence in $S$.
\end{coro}

\begin{theorem}\label{513}
Let $S$ a semitopological monoid with left shifts or right open shifts, then $\mathcal{C}_{2}(S)=S/\sim$, where $\sim$ is the  smallest closed congruence on $S$. 
\end{theorem}

\begin{proof}
Let $\sim$ be the smallest closed congruence on $S$ and  let $\pi\colon S\longrightarrow S$ be the respective quotient mapping.  Corollary \ref{04152} guarantees that $S/\sim$ is $T_{2}$, therefore there is a continuous mapping $g\colon \mathcal{C}_{2}(S)$, such that $g\circ \varphi_{(S,\mathcal{C}_{2})}=\pi$. By Proposition \ref{0157}, Theorem \ref{0156} and given what the $T_{2}$ spaces class is closed under  super topologies,  we have  $\mathcal{C}_{2}(S)=S/\simeq$, being  $\simeq$ a congruence  on $S$, which is closed by  Proposition \ref{04151}. Therefore $\sim\subseteq \simeq$, so we can define a continuous mapping $h\colon S/\sim \longrightarrow S/\simeq$, such that $h\circ \pi=\varphi_{(S,\mathcal{C}_{2})}$. It easily follows that $h$ is the inverse of $g$, this completes the proof.
\end{proof}

\begin{theorem}\label{015}
Let $S$ be a topological semigroup with open shifts, then $S_{sr}$ is a topological semigroup. Also if $S$ is a monoid, then $S$ is a  quasiregular monoid.
\end{theorem}

\begin{proof}
Let $S$ be a topological semigroup with open shifts. Let $U$ be  a  regular open in $S$ such that $ab\in U$ being $a,b\in S$. Given that the operation on $S$ is continuous, we can find open sets in $S$, $V$ and $W$, containing $a$ and $b$, respectively, holding $VW\subseteq U$. The continuity of the operation and the fact what $Int(Cl(V))Cl(W)$ is open in $S,$ imply   
\begin{equation}
\begin{split}
Int(\overline{V})Int(\overline{W})\subseteq Int(\overline{V} \overline{W})&=Int(Int (\overline{V}Cl(W)))\\&\subseteq Int(\overline{V}\overline{W})\\&\subseteq Int(\overline{VW})\\&\subseteq Int(\overline{U})=U.
\end{split}
\end{equation} Therefore $S_{sr}$ is a topological semigroup. Let us  suppose  $S$ is a monoid and let us prove that $S_{sr}$ is a topological  monoid $T_{3}$. Obviously $S_{sr}$ is a monoid and  we have already proved $S_{sr}$ is a topological semigroup, it rests to prove that it is $T_{3}$. Indeed let $U$ a regular open in $S$ and $x\in U$. Since the operation in $S$ is continuous and $ex=x$ we can find an open neighborhood of $e$, $V$, and a open neighborhood of $x$, $W$, such that $VW\subseteq U$. Therefore 
\begin{equation}
\begin{split}
x\in\overline{Int(\overline{W})}\subseteq \overline{W}\subseteq V\overline{W}=Int (V\overline{W})&\subseteq Int(\overline{V}\overline{W})\\&\subseteq Int (\overline{VW})\\&\subseteq Int(\overline{U})=U.
\end{split}
\end{equation}This proves that $S_{sr}$ is $T_{3}$.
\end{proof}
From similar argument to the proof of  Theorem 2.6 of \cite{T2} and  Theorem \ref{015} we obtain the following result.

\begin{coro} \label{075} If $S$ is a  topological monoid with open shifts, then $\mathcal{C}_{3}(S)=S_{sr.}$
\end{coro}

From  Proposition \ref{0151},  Theorem \ref{015}, and   Corollary \ref{075}, we have the following theorem.

\begin{theorem}\label{0153}
Let $S$ be a topological monoid with open  shifts, then $\mathcal{C}_{r}(S)=\mathcal{C}_{0}(S_{sr})$.
\end{theorem}

\section{\textbf{EPIREFLECTIONS PRESERVING PRODUCTS}.}

Given an epireflective subcategory of $TOP$, $\mathcal{C}$, we say that the epireflection induced by $\mathcal{C}$ preserves products in a subcategory $\mathcal{D}$ of $TOP$, if $\mathcal{C}(\prod_{i\in} X_{i})=\prod_{i\in I} \mathcal{C}(X_{i})$ for each family $\{X_{i}\}_{i\in I}$ of spaces in $\mathcal{D}$.

\begin{theorem}\label{5141}
Let $\mathcal{C}$ an epireflective subcategory of $TOP$ closed under super topologies satisfying  $\prod_{j\in J} X_{j}\in \mathcal{C}$ if and only if $X_{j}\in \mathcal{C}$ for each $j\in J$. Then the $\mathcal{C}$-epireflection  preserves products in the category of semitopological monoids with open shifts.
\end{theorem}

\begin{proof}
Let $\{S_{i}\}_{i\in I}$ be a family of semitoplogical semigroups with open shifts and let $\mathcal{C}$ an epireflective subcategory of $TOP$ closed under super topologies.  Proposition \ref{0131} and  Theorem \ref{0156} imply $\mathcal{C}(\prod_{i\in I}S_{i})=(\prod_{i\in I} S_{i})/R$, being $R$ a congruence in $\prod_{i\in I} S_{i}$. Given $k\in I$, let us define the following relation in $X_{k}$: $x\sim_{k} y$ if there exist $\textbf{x}, \textbf{y}\in \prod_{i\in I} S_{i}$ such that $\varphi_{(\mathcal{C},\prod_{i\in I} S_{i})}(\textbf{x})=\varphi_{(\mathcal{C},\prod_{i\in I} S_{i})}(\textbf{x})$ and $p_{k}(\textbf{x})=p_{i}(\textbf{y})$, being $p_{k}\colon \prod_{i\in I} S_{i}\longrightarrow S_{k}$, the $k$-th projection. It is easy to prove that $\sim_{k}$ is a congruence on $X_{k}$. Let $\pi_{k}\colon X_{k}\longrightarrow X_{k}/\sim_{k}$ the respective  quotient mapping. From definition of $\sim_{k}$, for each $k\in K$, we have $f\colon \prod_{i\in I}( S_{i}/\sim_{i})\longrightarrow (\prod_{i\in I}S_{i})/R$, given by $f((\pi_{i}(x_{i}))_{i\in I})=\varphi_{(\mathcal{C},\prod_{i\in I} S_{i})}((x_{i})_{i\in I}))$, is a bijection.  Proposition \ref{0415} implies $\pi_{i}$ is open for each $i\in I$, therefore $f$ is a homemorphism, the hypothesis over $\mathcal{C}$ guarantees that $S_{i}/\sim_{i}\in \mathcal{C}$, for each $i\in I$, therefore  for each $i\in I$ we can define a continuous mapping  $g_{i}\colon \mathcal{C}(S_{i})\longrightarrow S_{i}/\sim_{i}$, by $g_{i}(\varphi_{(\mathcal{C},S_{i})}(x))=\pi_{i}(x)$, for each $x\in S_{i}$. Therefore  $\prod_{i\in I} g_{i}\colon \prod_{i\in I} \mathcal{C}(S_{i})\longrightarrow \prod_{i\in I} (S_{i}/\sim_{i})$ is continuous, so that $f\circ \prod_{i\in I} g_{i}\colon \prod_{i\in I} \mathcal{C}(S_{i}))\longrightarrow \mathcal{C}(\prod_{i\in I}S_{i})$ is continuous. On the other hand, since $\prod_{i\in I} \mathcal{C}(S_{i})\in \mathcal{C}$, we can find a continuous mapping $k\colon \mathcal{C}(\prod_{i\in I}S_{i})\longrightarrow \prod_{i\in I}\mathcal{C}(S_{i})$, such that $k\circ \varphi_{(\mathcal{C},\prod_{i\in I} S_{i})}=\prod_{i\in I} \varphi_{(\mathcal{C}, S_{i})}$. It is easy to prove that $k$ and $f\circ \prod_{i\in I} g_{i}$, are inverse to each other. This completes the proof.
\end{proof}

It is known that the $\mathcal{C}_{i}$-reflection does not preserve products (see introduction of \cite{HU}), $i\in \{1,2,3,t\}$. However $\mathcal{C}_{i}$, $i\in\{0,1,2\}$, preserves products in the category of semitopological groups (propositions 3.3, 3.4 and 3.5 of \cite{T3}), while $\mathcal{C}_{3}$ and $\mathcal{C}_{r}$ preserve products in the category of paratopological groups (Proposition 3.6 of \cite{T3}). In the two following results we  give  similar results in the semitopological and topological monoids with open shifts.\\
Given $\mathcal{C}_{i}$, $i\in\{0,1,2\}$, satisfies the hypothesis
of  Theorem \ref{5141}, we have the following corollary.

\begin{coro}\label{5142} The $\mathcal{C}_{i}$-reflection preserves products in the category of semitopological semigroups with open shifts, $i\in\{0,1,2\}$.
\end{coro}

According to   Lemma 3 of \cite{semi}, we have the semiregularitation respects products, therefore from  Corollary \ref{075},  Theorem \ref{0153} and  Corollary \ref{5142}, we have the following theorem.

\begin{theorem}
The $\mathcal{C}_{i}$-reflection preserves products in the category of topological monoids with open shifts, $i\in\{3, r\}$.
\end{theorem}

\section{\textbf{THE CELLULARITY OF TOPOLOGICAL MONOIDS} .}

\begin{prop}\label{0154} $c(X)=c(\mathcal{C}_{0}(X))$ for each topological space $X$.
\end{prop}
\begin{proof}
Given that $\mathcal{C}_{0}(X)$ is a continuous image of $X$, we have $c(\mathcal{C}_{0}(X))\leq c(X)$. Let $\mathcal{U}$ a cellular family in $X$,  Proposition \ref{0152} guarantees that \\$\varphi_{(\mathcal{C}_{0},X)}^{-1}(\varphi_{(\mathcal{C}_{0},X)}(U))=U$ for each $U\in \mathcal{U}$, therefore $\{\varphi_{(\mathcal{C}_{0}, X)}(U):U\in\mathcal{U}\}$ is a cellular family in $\mathcal{C}_{0}(X)$, so that $c(X)\leq c(\mathcal{C}_{0}(X))$, this completes the proof.
\end{proof}

\begin{prop}\label{0155}$c(X)=c(X_{sr})$ for each topological space $X$.
\end{prop}

\begin{proof}
Given that $X_{sr}$ is a continuous image of $X$, we have $c(X_{sr})\leq X$. Let $\mathcal{U}$ a cellular family in $X$ and let $U$ and $V$ be in $\mathcal{U}$, therefore $U\cap V=\emptyset$. Since $U$ and $V$ are open sets in $X$, we have $int \overline{U}
\cap int\overline{V}=\emptyset$. Also since for all $U\in \mathcal{U}$ it holds $U\subseteq int\overline{U}$, we have $int\overline{U}\neq \emptyset$ for all $U\in U$. This proves that the $\{int\overline{U}: U\in \mathcal{U}\}$ is a cellular family in $X_{sr}$, therefore $c(X)\leq c(X_{sr})$,  this completes the proof. \end{proof}

From  Propositions \ref{0154} and \ref{0155} and the Theorem \ref{0153} we have the following corollary.

\begin{coro}\label{0157}Let $X$ be a quasiregular space, then $c(X)=c(\mathcal{C}_{r}(X)).$
\end{coro}

\begin{coro}\label{0158} If $X$ is quasiregular space, we have $c(X)=c(\mathcal{C}_{i}(X))$ for each $i\in \{0, 1, 2,3, r\}$.
\end{coro}

\begin{proof}Let $X$ be a quasiregular espace. By Proposition \ref{0154} and the definition of cellularity it is clear what $c(\mathcal{C}_{r}(X))\leq c(\mathcal{C}_{3}(X))$ and $c(\mathcal{C}_{r}(X))\leq c(\mathcal{C}_{2}(X))\leq c(\mathcal{C}_{1}(X))\leq c(\mathcal{C}_{0}(X))=c(X)$.  Corollary \ref{0157} guarantees that $c(X)=c(\mathcal{C}_{r}(X))$, this completes the proof.
\end{proof}

From  Theorem \ref{015} and the Corollary \ref{0158} we have the following Theorem.

\begin{theorem}\label{0751}Let $S$ be a topological monoid with open shifts. Then the following statements are equivalent,

\begin{itemize}
\item[i)]$S$ has cellularity countable.
\item[ii)]$\mathcal{C}_{0}(S)$ has cellularity countable.
\item[iii)]$\mathcal{C}_{1}(S)$ has cellularity countable.
\item[iv)]$\mathcal{C}_{2}(S)$ has cellularity countable.
\item[v)]$\mathcal{C}_{3}(S)$ has cellularity countable.
\item[vi)]$\mathcal{C}_{r}(S)$ has cellularity countable.
\end{itemize}
\end{theorem}

\begin{lemma}\label{510}If $S$ is  a  cancellative topological semigroup with open shifts, then  $\mathcal{C}_{0}(S_{sr})$ is cancellative. 
\end{lemma}

\begin{proof}Let $S$ be a cancellative topological semigroup with open shifts and let us see what $\mathcal{C}_{0}(S_{sr})$ is cancellative. Indeed, let us suppose $\varphi_{(\mathcal{C}_{0},S_{sr})}(cx)=\varphi_{(\mathcal{C}_{0},S_{sr})}(cy)$ but $\varphi_{(\mathcal{C}_{0},S_{sr})}(x)\neq \varphi_{(\mathcal{C}_{0},S_{sr})}(y)$. From  Proposition \ref{0152} we can find an open regular in $S_{sr}$, $U$, that without loss of generality we can assume that $x\in U$ and $y\notin U$. Since $l_{c}\colon S\longrightarrow cS$ is an homemorphism, we have $cU$ is open regular in $cS$. There exist an open set $V$,  in $S$ such that $cU=V\cap cS$. Given that $cS$ is open in $S$, we have $ cx\in cU=int_{cS}(Cl_{cS}(cU))=cS \cap int_{S}( Cl_{S}(V))$. Since $S$ is cancellative, $cy\notin cU$ and  and in consequence $cy\notin int_{S}(Cl_{S}(V)$, but $int_{S}(Cl_{S}(V)$ is open in $S_{sr}$, from Proposition \ref{0151} we can say $\varphi_{(\mathcal{C}_{0},S_{sr})}(cx)\neq \varphi_{(\mathcal{C}_{0},S_{sr})}(cy)$ obtaining a contradiction, this implies that $\varphi_{(\mathcal{C}_{0},S_{sr})}(x)\neq \varphi_{(\mathcal{C}_{0},S_{sr})}(y)$ and therefore $\mathcal{C}_{0}(S_{sr})$ is cancellative to the  left, proving that is cancellative to the right is analogues.
\end{proof}

In \cite{T3} M. Tkachenko proved that the $\sigma$-compact paratopological groups have countable cellularity. In the following two  theorems we give   analogues results for topological semigroups, but given that we do not have group operation, we have changed the $\sigma$-compactness for compactness in the first theorem, and in the second theorem, in addition to the $\sigma$-compactness we have added the sequential  compactness.

\begin{theorem}\label{5132}Let $S$ be a compact topological monoid cancelative with open shifts, then $S$ has countable cellularity.
\end{theorem}

\begin{proof}
Let $S$ be a compact topological monoid cancelative with open shifts, from  Proposition \ref{0151},  Corollary \ref{0354},  Theorem \ref{015} and Lemma \ref{510}, we have $\mathcal{C}_{r}(S)=\mathcal{C}_{0}(S_{sr})$ is a cancellative topological monoid which is compact. Since $S$ is compact,   Theorem 2.5.2 of \cite{T1} implies that $\mathcal{C}_{r}(S)$ is a compact topological group and from Corollary 2.3 of \cite{T3} we have $\mathcal{C}_{r}(S)$ has countable cellularity. Finally, applying   Theorem \ref{0751} we have $S$ has countable cellula\-rity.
\end{proof}

\begin{theorem}\label{5131}If $S$ is a $\sigma$-compact and sequentially compact cancellative topological  monoid  with open shifts, then $S$ has countable celullarity. 
\end{theorem}

\begin{proof}
Let $S$ be a $\sigma$-compact and sequentially compact cancellative topological  monoid  with open shifts.  Lemma \ref{510} guarantees that $\mathcal{C}_{r}(S)$ is cancellative, also being continuos image of $S$, we have $\mathcal{C}_{r}(S)$ is $\sigma$-compact and sequentially compact.  Theorem 6 of \cite{can} implies that $\mathcal{C}_{r}(S)$ is a $\sigma$-compact topological group and from Corollary 2.3 of \cite{T3}, we have $\mathcal{C}_{r}(S)$ has countable celullarity. If we apply  Theorem \ref{0751} we have $S$ has countable celullarity.
\end{proof}

\end{document}